\documentclass{article}
\usepackage{amsmath}
\usepackage{amsfonts, amssymb, amsthm}
\usepackage{epsfig}
\usepackage{color}
\usepackage{multicol}

\usepackage[colorlinks=true]{hyperref}
\hypersetup{urlcolor=blue, citecolor=red}

\newlength{\hchng}
\newlength{\vchng}
\newtheorem{thm}{Theorem}[section]

\newtheorem{lemma}[thm]{Lemma}

\newtheorem{preremark}[thm]{Remark}
\newenvironment{remark}{\begin{preremark}\rm}{\medskip \end{preremark}}
\numberwithin{equation}{section}

\newcommand{\norm}[1]{\left\Vert#1\right\Vert}
\newcommand{\abs}[1]{\left\vert#1\right\vert}

\newcommand{\R}{\mathbb R}
\newcommand{\ep}{\varepsilon}

\newcommand{\grad} {\nabla}
\newcommand{\lap} {\triangle}

\newcommand{\dd} {\; \mathrm{d}}

\DeclareMathOperator*{\osc}{osc}

\DeclareMathOperator{\dv}{div}

\newcommand{\Q}{Q}

\title{H\"older estimates for advection fractional-diffusion equations}

\author{Luis Silvestre}

\begin{document}
\maketitle
\begin{abstract}
We analyse conditions for an evolution equation with a drift and fractional diffusion to have a H\"older continuous solution. In case the diffusion is of order one or more, we obtain H\"older estimates for the solution for any bounded drift. In the case when the diffusion is of order less than one, we require the drift to be a H\"older continuous vector field in order to obtain the same type of regularity result.
\end{abstract}

\section{Introduction}
This papers concerns the solutions of an equation of the form
\begin{equation} \label{e:main}
u_t + b \cdot \grad u + (-\lap)^s u = f.
\end{equation}

We will prove that the solution $u$ becomes immediately H\"older continuous. In the case $s \in [1/2,1)$, all we require for the vector field $b$ is that it is bounded. In the case $s \in (0,1/2)$, we require $b$ to be in the H\"older class $C^{1-2s}$.

The case $s =1/2$ is often referred to as the critical case because the diffusion is an operator of the same order as the advection. Both terms in \eqref{e:main} have the same weight at all scales and no perturbation techniques can be applied to obtain regularity of the solution $u$. This critical case was already studied in \cite{silvestre2009differentiability}, where the result was used to establish the existence of classical solutions for the Hamilton-Jacobi equation with critical fractional diffusion. The main ideas in this paper originated in that work. In future work, we plan to apply the result of this paper to other nonlinear equations like for example conservations laws with fractional diffusion.

It is interesting to compare this result in the critical case $s=1/2$ with the result in \cite{caffarelli2006drift}. In that paper the same equation is studied and the H\"older continuity result is obtained using De Giorgi's approach for parabolic equations. The only difference is that the assumptions for $b$ in \cite{caffarelli2006drift} are that it belongs to the BMO class and it is divergence free. In this paper (or in \cite{silvestre2009differentiability}) we do not assume $b$ to be divergence free, but we need the slightly stronger assumption that $b \in L^\infty$ instead of just belonging to the BMO class. The methods used in the proofs are of a different nature. In \cite{caffarelli2006drift}, thanks to the assumption that $\dv b = 0$, they use variational techniques (in the style of De Giorgi), whereas in this article our techniques are purely non-variational and are ultimately based on quantitative comparison principles which are arguably simpler. An alternative approach is used in \cite{kiselev2010variation} where the H\"older continuity is obtained through the study of the dual equation. They still require that $\dv b = 0$ in order for the dual equation to have the same advection-diffusion form.

The case $s<1/2$ is known as the supercitical case because the diffusion is of lower order than the advection. This means that in small scales the drift term will be stronger than the diffusion which seems to suggest that it would not be possible to prove any regularity for $u$. However, if we assume a H\"older modulus of continuity for $b$, we can compensate this deterioration of our control of the drift term in small scales by a change of variables following the flow of $b$. The supercitical case is the main focus of this article.

It is interesting to compare our result in the case $s<1/2$ with the result of \cite{MR2483817}. In that article, Constantin and Wu explore the applicability of Caffarelli-Vasseur \cite{caffarelli2006drift} techniques to the supercritical regime. They need to assume that $b \in C^{1-2s}$ and $\dv b = 0$. Therefore, in the supercitical regime, our non-variational method provides a strictly better result than the De Giorgi's approach, since the regularity assumption on $b$ is the same, but we do not need to assume $\dv b = 0$. This result is used in \cite{MR2482996} to obtain the well posedness of a modified supercitical quasi-geostrophic equation.

The case $s > 1/2$ is often referred to as the subcritical case because the diffusion is of higher order than the advection in \eqref{e:main}. The method in this paper provides a H\"older continuity result for $u$ if $b \in L^\infty$ in the subcritical case. This condition is far from sharp, since one could obtain the same results under more general conditions. In \cite{MR2283957}, heat kernel estimates are obtained using perturbation techniques for $s \in (1/2,1)$ for vector fields which belong just to a Kato class. In \cite{silvestre2010polynomial}, H\"older estimates are obtained in the case $s \in (1/2,1)$ but for $b \cdot \grad u$ replaced by a term with superlinear growth respect to $\grad u$. The subcritical case is not the main objective of this paper. We still include it in our main theorem to show what result the method provides.

Here is our main theorem.

\begin{thm}\label{t:main}
Let $u$ be a bounded function in $\R^n \times [0,1]$ which solves the equation \eqref{e:main}. Assume that $s \in (0,1)$, $f \in L^\infty$. If $s \in [1/2,1)$, we assume $b \in L^\infty$. If $s \in (0,1/2)$, we assume $b \in C^{1-2s}$. Then $u$ is H\"older continuous for positive time $t \in (0,1]$. Moreover, there is an estimate
\[ |u(x,t) - u(y,r)| \leq C \frac{ |x-y|^\alpha + |t-r|^{\alpha/(2s)}}{t^{\alpha/(2s)}} (\norm{u}_{L^\infty} + \norm{f}_{L^\infty})\]
for every $x,y \in B_{1/2} \text{ and } 0\leq r \leq t \leq 1$, where the constants $C$ and $\alpha$ depend on $s$, $n$ and $\norm{b}_{C^{1-2s}}$ but not on $u$ or $f$.
\end{thm}

It is our intention to keep this article as simple and accessible as possible. The fractional Laplacian could be replaced by arbitrary integro-differential operators of order $2s$, even with discontinuous coefficients, as in \cite{silvestre2009differentiability}, but we chose to keep only the fractional Laplacian in order to make the article accessible to a larger audience. Moreover, we will concentrate on proving the a priori estimate for Theorem \ref{t:main}, assuming that we have a classical solution. In the last section, we will discuss the notions of weak solution and how the proof would adapt to those scenarios.

\section{Scaling}
\label{s:scaling}
If $u$ satisfies the equation
\[ u_t + b \cdot \grad u + (-\lap)^s u = f \]
then the function $u^\lambda = \lambda^{-\alpha} u(\lambda x,\lambda^{2s} t)$ satisfies
\[ u^\lambda_t + \lambda^{2s-1} b \cdot \grad u^\lambda + (-\lap)^s u^\lambda = \lambda^{2s-\alpha} f^\lambda \]
with $f^\lambda(x,t) = f(\lambda x, \lambda^{2s} t)$.

We define the parabolic cylinders in terms of the scaling of the equation:
\[ \Q_r := B_r \times [-r^{2s},0]. \]

\section{The diminish of oscillation lemma}
\label{s:dimish-of-oscillation}

The proof of the diminish of oscillation lemma uses essentially the same ideas as the proof in \cite{silvestre2009differentiability} for the case $s=1/2$. The generalization to the full range $s \in (0,1)$ does not present major difficulties. In this section we present the proof in full detail.

\begin{lemma}[point estimate] \label{l:point-estimate}
Let $u \leq 1$ in $\R^n \times [-2,0]$ and assume it satisfies the following inequality in $B_R \times [-2,0]$.
\begin{equation} \label{e:ineq}
u_t - A |\grad u| + (-\lap)^s u \leq \ep_0 
\end{equation}
Where $R=2+A$. Assume also that 
\[ \abs{\{u \leq 0\} \cap ( B_1 \times [-2 , -1]) } \geq \mu > 0. \]
Then, if $\ep_0$ is small enough there is a $\theta>0$ such that $u \leq 1-\theta$ in $B_1 \times [-1 , 0]$.

\noindent (the maximal value of $\ep_0$ as well as the value of $\theta$ depend only on $s$, $\mu$ and $n$)
\end{lemma}

\begin{remark}
Note that if for some bounded vector field $b$, the function $u$ satisfies the equation
\begin{equation} \label{e:equation-for-u}
 u_t + b \cdot \grad u + (-\lap)^s u = f,
\end{equation}
then $u$ also satisfies the inequality \eqref{e:ineq} with $A = \norm{b}_{L^\infty}$.
\end{remark}

\begin{proof}
Let $m: [-2 , 0] \to \R$ be the function such that:
\begin{equation} \label{e:ODEform}
\begin{aligned}
m(-2) &= 0, \\
m'(t) &= c_0 | \{x \in B_1: u(x,t) \leq 0\}| - C_1 m(t).
\end{aligned}
\end{equation}

The solution of this ODE for $m$ can be computed explicitly
\[ m(t) = \int_{-2}^t c_0 | \{x : u(x,s) \leq 0\} \cap B_1 | e^{-C_1
(t-s)} \dd s. \]

We will show that if $c_0$ is small and $C_1$ is large, then
$u \leq 1 - m(t) + 2 \ep_0$ in $B_1 \times [-1,0]$. 
Since for $t \in [- 1 ,0]$,
\[ m(t) \geq c_0 e^{- 2C_1  } |\{ u \leq 0 \} \cap B_1 \times
[-2  , -1 ] | \geq  c_0 e^{-2C_1} \mu.\]We set $\theta = c_0 e^{- 2C_1 }\mu/2$ for $\ep_0$ small and finish the proof of the Lemma. 

Let $\beta : \R \to \R$ be a fixed smooth non increasing function such that $\beta(x)=1$ if $x \leq 1$ and $\beta(x)=0$ if $x \geq 2$.

Let $\eta(x,t) = \beta(|x|+ A t) = \beta (|x| - A |t|)$. As a function of $x$, $\eta(x,t)$ looks like a bump function for every fixed $t$. The bump shrinks at speed $A$. Recall that $A$ will be the $L^\infty$ bound for the vector field in the drift term, so $\eta$ is a bump that is shrinking at the maximum possible speed.

The strategy of the proof is to show that the function $u$ stays below $1 - m(t)\eta(x,t) + \ep_0 (2+t)$. 

Note that if there was no diffusion, $1 - m \eta(x,t) + \ep_0 (2+t)$ would be a supersolution of the equation for any fixed constant $m$. We would be able to prove that $u<1-\theta$ in $Q_1$ only in the case that $u(x,-2) < 1 - m \eta(x,-2)$. The diffusion is the reason why we can use a variable value of $m$ and obtain an actual improvement of oscillation.

At those points where $\eta = 0$ (precisely where $|x| \geq 2 - A t = 2 + A |t|)$, $(-\lap)^s \eta < 0$. Since $\eta$ is smooth, $(-\lap)^s \eta$ is continuous and it remains negative for $\eta$ small enough. Thus, there is some constant $\beta_1$ such that $(-\lap)^s \eta \leq 0$ where $\eta \leq \beta_1$.

In order to arrive to a contradiction, we assume that $\eta(x,t) > 1 - m(t) + \ep_0 (2+t)$ for some point $(x,t)
\in B_1 \times [-1,0]$. We look at the maximum of the function
\[ w(x,t) = u(x,t) + m(t) \eta(x,t) - \ep_0 (2+t). \]
Since we are assuming that there is one point in $B_1 \times [-1,0]$ where $w(x,t) > 1$, $w$ must be larger than $1$ at the point that realizes the maximum of $w$. Let $(x_0,t_0)$ be the point where this maximum is realized. Since $w(x_0,t_0)>1$, the point $(x_0,t_0)$ must belong to the support of $\eta$. Thus $|x_0| < R$. 

Since $w$ achieves its maximum at $(x_0,t_0)$, then at that point we have $w_t \geq 0$ and $\grad w = 0$. Therefore, $u_t \geq -m' \ \eta -m \ \eta_t + \ep_0$ and $\grad u = -m \grad \eta$. We will now show that the equation for $u$ cannot hold at $(x_0,t_0)$.

We start by estimating $u_t(x_0,t_0)$.
\begin{equation} \label{e:bound-for-vt}
\begin{aligned}
u_t(x_0,t_0) &\geq -m'(t_0) \eta(x_0,t_0) - m(t_0) \partial_t \eta(x_0,t_0) + \ep_0 \\
&= -m'(t_0) \eta(x_0,t_0) + m(t_0) A |\grad \eta(x_0,t_0)|^{q} + \ep_0
\end{aligned}
\end{equation}

We also note that 
\begin{equation} \label{e:bound-for-grad}
|\grad u(x_0,t_0)| = m(t) |\grad \eta(x_0,t_0)|.
\end{equation}

The measure of the set $| \{x : u(x,s) \leq 0\} \cap B_1 |$ is used in the estimation of $(-\lap)^s u$. The computations below are for fixed $t=t_0$, so we omit writing the time $t$ in order the keep the computations cleaner.

Since $u+m \eta$ attains its maximum at $x_0$, $u(x_0) - u(x_0+y) \geq -m \, (\eta(x_0)-\eta(x_0+y))$ for all $y$. Replacing this inequality in the formula for $(-\lap)^s u$ we can easily obtain $(-\lap)^s u(x_0) \geq  -m \ (-\lap)^s \eta(x_0)$. However, we can improve this estimate in terms of the value $|\{u \leq 0\} \cap B_1|$ (this is the key idea from \cite{silvestre2009differentiability}).

Let $y \in \R^n$ be such that $u(x_0+y)\leq 0$. We estimate $u(x_0) + m \, \eta(x_0) - u(x_0+y) - m \, \eta(x_0+y)$ using that $u(x_0)+m \eta(x_0) = w(x_0,t_0) + \ep_0 (1+t_0) > 1$.
\[ u(x_0) + m \, \eta(x_0) - u(x_0+y) - m \, \eta(x_0+y) \geq 0 - m + 1 \]
We choose $c_0$ small so that $m \leq 1/2$ and
\begin{equation} \label{e:est-good-set}
u(x_0) + m \, \eta(x_0) - u(x_0+y) - m \, \eta(x_0+y) \geq \frac 1 2
\end{equation}

Now we estimate $(-\lap)^s u(x_0,t_0)$, we start writing the integral
\begin{align*}
(-\lap)^s u(x_0,t_0) &= \int_{\R^n} \frac{ u(x_0) - u(x_0+y)}{|y|^{n+2s}} \dd y \\
\intertext{We estimate $u(x_0) - u(x_0,y)$ by below with $\eta(x_0+y)-\eta(x_0)$ except at those points where $x_0+y$ is in the \emph{good} set $G := \{u \leq 0\} \cap B_1$ where we use \eqref{e:est-good-set}}
&\geq -m(t_0) (-\lap)^s \eta(x_0,t_0) \\ & \hspace{.5in} + \int_{G} \frac{ (u + m \eta) (x_0) - (u + m \eta) (x_0+y)}{|y|^{n+2s}} \dd y \\
&\geq -m(t_0) (-\lap)^s \eta(x_0,t_0) + \int_{G \setminus B_r} \frac{1/2}{|y|^{n+2s}} \dd y \\
&\geq -m(t_0) (-\lap)^s \eta(x_0,t_0) + c_0 |G \setminus B_r|
\end{align*}
for some universal constant $c_0$ (this is how $c_0$ is chosen in \eqref{e:ODEform}).

We consider two cases and obtain a contradiction in both. Either $\eta(x_0,t_0) > \beta_1$ or $\eta(x_0,t_0) \leq \beta_1$.

Let us start with the second. If $\eta(x_0,t_0) \leq \beta_1$, then $(-\lap)^s \eta(x_0,t_0) \leq 0$, then
\begin{equation} \label{e:bound-case1}
(-\lap)^s u(x_0,t_0) \geq c_0 |G \setminus B_r|
\end{equation}

Replacing \eqref{e:bound-for-vt}, \eqref{e:bound-for-grad} and \eqref{e:bound-case1} into \eqref{e:equation-for-u} we obtain
\[ \ep_0 \geq -m'(t_0) \eta(x_0,t_0) + \ep_0 + c_0 |G| \]
but for any $C_1>0$ this contradicts \eqref{e:ODEform}.

Let us now analyze the case $\eta(x_0,t_0) > \beta_1$. Since $\eta$ is a smooth, compactly supported function, there is some constant $C$, such that $| (-\lap)^s \eta| \leq C$.
Then we have the bound
\[
(-\lap)^s u(x_0,t_0) \geq -m(t_0) (-\lap)^s \eta(x_0,t_0) + c_0 |G| \geq -C m(t_0) - c_0 |G \setminus B_r|
\]
Therefore, replacing in \eqref{e:equation-for-u}, we obtain
\[ \ep_0 \geq -m'(t_0) \eta(x_0,t_0) + \ep_0 + c_0 |G| - C m(t_0)\]

and we have
\begin{equation*} 
- m'(t_0) \eta(x_0,t_0) - C m(t_0) + c_0 |G| \leq 0.
\end{equation*}

We replace the value of $m'(t_0)$ in the above inequality using
\eqref{e:ODEform}
\[  (C_1 \eta(x_0,t_0)- C) m(t_0) + c_0 (1- \eta(x_0,t_0)) |G| \leq 0. \]
Recalling that $\eta(x_0,t_0) \geq \beta_1$, we arrive at a contradiction if $C_1$ is chosen large enough.
\end{proof}

The following Lemma is a crucial step in the proof of H\"older continuity. It says that if a solution of the equation in the unit cylinder $Q_1 = B_1 \times [-1,0]$ has oscillation one, then its oscillation in a smaller cylinder $\Q_{r}$ is less than a fixed constant $(1-\theta)$. This type of Lemmas are very common when proving H\"older regularity results. Sometimes they are called \emph{growth} lemmas, since they say that the oscillation in a cylinder of a solution $u$ necessarily grows when considering a larger cylinder. For nonlocal equations one needs to add an extra condition giving a bound to the oscillation of the function $u$ outside $Q_1$.

\begin{lemma}[diminish of oscillation] \label{l:diminish-of-oscillation}
Assume that $u$ satisfies the following equation in $\Q_1$
\[ u_t - b \cdot \grad u +(-\lap)^s u = f \label{e:d} \]
Where $b$ is a bounded vector field. There are universal constants $\theta>0$, $\alpha>0$ and $\ep_0>0$ (depending on $||b||_{L^\infty}$, the dimension $n$ and $s$) such that if
\begin{align*}
||f||_{L^\infty(Q)} &\leq \ep_0 \\
\osc_{Q_1} u &\leq 1  \\
\osc_{B_M \times [-1,0]} u &\leq |M/r|^\alpha \qquad \text{for all } M > 1.
\end{align*}
where $r$ is a constant depending on $\norm{b}_{L^\infty}$ and $s$. Then
\[ \osc_{\Q_r} u \leq (1-\theta) \]
\end{lemma}

\begin{proof}
Let us choose a constant $R>1$ so that $R = 2+||b||_{L^\infty}$ if $s \geq 1/2$ or $R = R^{1-2s} \norm{b}_{L^\infty} + 2$ if $s \leq 1/2$.

We consider the rescaled version of $u$:
\[ \tilde u(x,t) = u(x / (2R), t / (2R)^{2s} ). \]

The function $\tilde u$ satisfies the equation
\[ \tilde u_t + (2R)^{1-2s} b \cdot \grad \tilde u + (-\lap)^s \tilde u = (2R)^{-2s} f. \]
In $B_{2R} \times (-2R,0]$. We want to apply Lemma \ref{l:point-estimate} to $\tilde u$ minus a constant. Note that either with $s \in [1/2,1)$ or $s \in (0,1/2)$ the vector field $(2R)^{1-2s}b$ is bounded by $A = (2R)^{1-2s} \norm{b}_{L^\infty}$.

Let $m = (\sup_{Q_1} u + \inf_{Q_1} u)/2$ be the median value of $u$ in $Q_1$ (and also the median value of $\tilde u$ in $Q_{2R}$).

The function $\tilde u$ will stay either above $m$ of below $m$ in half of the points in $B_1 \times [-2,-1]$ (in measure). More precisely, either $\{\tilde  u \leq m \} \cap (B_1 \times [-2,-1]) \geq |B_1|/2$ or $\{\tilde u \geq m \} \cap (B_1 \times [-2,-1]) \geq |B_1|/2$. Let us assume the former, otherwise we can repeat the proof with $(\sup_{Q_1} u -\tilde u)$ instead of $\tilde u$.

We will conclude the proof as soon as we can apply Lemma \ref{l:point-estimate} to $2(\tilde u-m)$. The only hypothesis we are missing is that $2(\tilde u-m)$ is not bounded above by $1$ outside $B_{2R}$. So we have to consider $v = \min(1,2(\tilde u-m))$ instead, and estimate the error in the right hand side of the equation. We prove that if $\alpha$ is small enough, then $v$ satisfies
\begin{equation} \label{e:kk}
 v_t - A |\grad v| + (-\lap)^s v \leq \tilde \ep_0
\end{equation}
for a small $\tilde \ep_0$ so that we can apply Lemma \ref{l:point-estimate}.

Note that inside $\Q_{2R}$, $\tilde u \leq 1$, thus $v = 2(\tilde u-m)$. The error in the equation in $\Q_{R}$ comes only from the tails of the integrals in the computation of $(-\lap)^s v$. Indeed $v_t=u_t$ and $\grad u= \grad v$. Let us estimate $(-\lap)^s v - (\lap)^s \tilde u$ in $Q_R$.

We choose $\ep_0$ smaller than $\tilde \ep_0$ (the small constant from Lemma \ref{l:point-estimate}) such that 
\begin{align*}
(-\lap)^s v(x,t) - (-\lap)^s \tilde u(x,t) 
&= \int_{x+y \notin B_{2R}} (\tilde u(x+y) - \sup_{Q_1} u)^+ \frac{\dd y}{|y|^{n+2s}} \\
&\leq \int_{x+y \notin B_{2R}} (((2R)^2 |x+y|)^\alpha - 1)^+ \frac{\dd y}{|y|^{n+2s}} \leq \tilde \ep_0 - \ep_0
\end{align*}
if $\alpha$ is chosen small enough.

Thus, $v$ satisfies \eqref{e:kk}. We can apply Lemma \ref{l:point-estimate} to $v$ and conclude the proof with $r:=1/(2R)$.
\end{proof}

\begin{remark}
Note that the diminish of oscillation lemma does not require the vector field $b$ to be H\"older continuous but only bounded. The H\"older continuity of $b$ will be required to prove the H\"older continuity of the solution $u$ if $s <1/2$ because Lemma \ref{l:diminish-of-oscillation} needs to be a applied at all scales. Thus, the regularity of $b$ compensates the bad scaling of the equation. This observation helps obtain some regularity results for some nonlinear supercitical equations in \cite{silvestre2009eventual} and \cite{MR2600693}.
\end{remark}

\section{H\"older regularity}
\label{s:c1a-regularity}

In this section we prove our main result about H\"older continuity. Note that the result stated in this section requires the function to solve the equation only in a cylinder in order to have H\"older continuity in a smaller cylinder. This is a local result, from which the global result of Theorem \ref{t:main} immediately follows.

For the case $s \in [1/2,1)$, the proof is a simple adaptation of the proof in \cite{silvestre2009differentiability}. For $s \in (0,1/2)$ we need to make a change of variables at the beginning following the flow of vector field $b$ in order to obtain the extra cancellation which allows us to use the hypothesis $b \in C^{1-2s}$.

\begin{thm} \label{t:main2}
Let $u$ be a bounded function in $\R^n \times [-1,0]$ which satisfies the following equation in $B_1 \times [-1,0]$
\begin{equation} \label{e:eq}
u_t - b \cdot \grad u + (-\lap)^s u = f
\end{equation}

There there is an $\alpha>0$ such that the function $u$ is $C^\alpha$ at $(0,0)$. Moreover we have the estimate
\begin{equation} \label{e:calpha}
|u(x,t) - u(0,0)| \leq C (|x|^\alpha + |t|^{\alpha/(2s)})(\norm{u}_{L^\infty} + \norm{f}_{L^\infty}) \qquad \text{for every $x \in \R^n$ and $-1 \leq t \leq 0$}
\end{equation}

The constants $C$ and $\alpha$ depend on $s$,  $||b||_{L^\infty}$ and the dimension $n$ only.
\end{thm}

\begin{proof}
For any $(x_0,t_0)$, we consider the normalized function
\[ v(x,t) = \frac 1 {2\norm{u}_{L^\infty}+\norm{f}_{L^\infty}/\ep_0} \, u(x,t). \]
where $\ep_0$ is the constant from Lemma \ref{l:diminish-of-oscillation}.

We prove the $C^\alpha$ estimate \eqref{e:calpha} by proving a $C^\alpha$ estimate for $v$ at $(0,0)$. Note that $\osc_{R^n \times [-1.0]} v \leq 1$. Moreover, $v$ is also a solution of an equation like \eqref{e:eq} and the corresponding right hand side $f$ has $L^\infty$ norm less than $\ep_0$

In the supercitical case $s < 1/2$, we need to compensate the bad scaling of the drift term with the required H\"older regularity of $b$. In order to take advantage of this, it is convenient to have $b(0,t)=0$, so that the H\"older continuity translates into a decay in $b$ close to the origin. Let us make a change of variables $\tilde v(x,t) = v(x+A(t),t)$, where $A(t)$ is a solution of the ODE:
\begin{align*}
A(0) &= 0 \\
A'(t) &= b(A(t),0).
\end{align*}
Note that the assumption that $b$ is H\"older continuous does not guarantee the uniqueness of the curve $A(t)$ following the flow of $b$. However, just from the continuity of $b$ there exists one solution $A$. The uniqueness of this solution is not relevant for the proof below.

Thus, $\tilde v$ satisfies the equation
\[\tilde v_t + \tilde b \cdot \grad \tilde v + (-\lap)^s \tilde v = f(x+A(t),t).\]
Where $\tilde b(x,t) = b(x+A(t),t) - b(A(t),t)$ is a H\"older continuous vector field such that $|\tilde b(x,t)| \leq C |x|^{1-2s}$. Note that $(x,t) \mapsto (a+A(t),t)$ is a Lipschitz change of variables, so the H\"older regularity of $v$ and $\tilde v$ are the same. From this point on we will write $v$ and $b$ but we mean $\tilde v$ and $\tilde b$ if we are considering the supercritical case. For the critical or subcritical case, no change of variables is necessary.

Let $r \in (0,1)$ be as in Lemma \ref{l:diminish-of-oscillation}. We will prove the following exponential decay of the oscillation in cylinders for some $\alpha>0$ to be chosen later.
\begin{equation} \label{e:oscillation-in-cylinders}
\osc_{\Q_{r^k}} v \leq r^{\alpha k} \qquad \text{for } k=0,1,2,3,\dots
\end{equation}
The H\"older estimate follows immediately from \eqref{e:oscillation-in-cylinders}.

We prove \eqref{e:oscillation-in-cylinders} by induction in $k$. Indeed, for $k=0$ it is known to be true from the hypothesis $\osc_{\R^n \times [-1,0]} v \leq 1$. Let us assume that it holds until certain value of $k$. We scale again by considering
\[ w(x,t) = r^{-\alpha_0 k} v(r^k x,r^{2sk} t). \]

As it was discussed in section \ref{s:scaling}, $w$ satisfies the equation
\[ w_t - r^{k \alpha (2s-1)} b(r^k x, r^{2sk} t) \cdot \grad w + (-\lap)^s w = r^{k(2s-\alpha)} f
\]
Since we are choosing $\alpha \leq 2s$, $||r^{k(2s-\alpha)} f||_{L^\infty} \leq \ep_0$. 

In the case $2s \geq 1$, $r^{k \alpha (2s-1)} \leq 1$, so the $L^\infty$ norm of the drift term stays bounded. In the case $2s < 1$, we use the H\"older decay of $b$ to get a uniform $L^\infty$ bound for $r^{k \alpha (2s-1)} b(r^k x, r^{2sk} t)$ for $|x| \leq 1$.

From the inductive hypothesis, we also have 
\begin{align}
\osc_{Q_1} w &\leq 1  \label{e:aa1}\\
\osc_{Q_{r^{-k}}} w &\leq r^{-\alpha_0 k} &&  \text{for } k=1,2,\dots  \text{ therefore } \osc_{B_M \times [-1,0]} w(x,t) \leq |M/r|^{\alpha_0}  \text{ for any } M>1. \label{e:aa2}
\end{align}

If $\alpha$ is small enough, we can apply Lemma \ref{l:diminish-of-oscillation} to obtain $\osc_{\Q_r} w \leq 1-\theta$. So, if $\alpha$ is chosen smaller than the one of Lemma \eqref{l:diminish-of-oscillation} and also so that $1-\theta \leq r^{\alpha}$, we have $\osc_{\Q_r} w \leq r^{\alpha}$, which means $\osc_{\Q_{r^{k+1}}} v \leq r^{\alpha (k+1)}$, which finished the proof by induction.

We have finished the proof of \eqref{e:oscillation-in-cylinders} for some small $\alpha>0$ depending on $s$, $||b||_{L^\infty}$ and $n$.
\end{proof}

\section{Weak solutions and generalizations}

In order to keep this article short and simple, we did not state the theorems in their most general form. We analyze in this last section the possible generalizations and the applicability of the result to weak solutions of \eqref{e:main}.

\subsection{Viscosity solutions}
When $b$ is a continuous vector field, the solution $u$ of \eqref{e:main} can be understood in the viscosity sense. Note that in the supercitical case $s < 1/2$, we assume that $b$ is H\"older continuous. 

By definition, a lower semicontinuous function $u$ is a supersolution of \eqref{e:main} in a set $\Omega$ if every time there is a smooth function $\varphi$ such that for some $(x_0,t_0) \in \Omega$ and $\rho>0$, $\varphi(x_0,t_0) = u(x_0,t_0)$ and $\varphi(x,t) \leq u(x,t)$ for all points $(x,t)$ such that $|x-x_0|\leq \rho$ and $t_0 - \rho \leq t \leq t_0$ (we say $\varphi$ touches $u$ from below at $(x_0,t_0)$), then the function
\[ v(x,t) = \begin{cases}
\varphi(x,t) & \text{if } |x-x_0|< \rho \text{ and } |t-t_0|<\rho \\
v(x,t) & otherwise
\end{cases}\]
satisfies the inequality $v_t(x_0,t_0) + b(x_0,t_0) \cdot \grad v(x_0,t_0) + (-\lap)^s v(x_0,t_0) \geq f(x_0,t_0)$.

Correspondingly, an upper semicontinuous function $u$ is a subsolution if every time $\varphi$ touches $u$ from above, the above constructed function $v$ satisfies the opposite inequality. A continuous function $u$ is said to be a solution when it is a subsolution and a supersolution at the same time. For more details on the definition and properties of viscosity solutions for integro-differential equations see \cite{barles2008second} or \cite{caffarelli2009regularity}.

The definition of viscosity solution is appropriate for elliptic or parabolic equations that satisfy a comparison principle. As we can see in the definition, it requires that the operator can be evaluated in smooth test functions. Hence, the vector field $b$ needs to be continuous. On the other hand, the definition of viscosity solutions does not require the equation to be linear. Thus, in the critical and subcritical case, the equation \eqref{e:main} can be replaced by the couple of inequalities
\begin{align*}
u_t + A |\grad u| + (-\lap)^s u &\geq B \\
u_t - A |\grad u| + (-\lap)^s u &\leq B
\end{align*}
where $A = \norm{b}_{L^\infty}$ and $B = \norm{f}_{L^\infty}$.

The proofs presented in this article follow naturally the viscosity solution spirit. Indeed, the adaptation of the proof of Lemma \ref{l:point-estimate} is straight forward since the function $w$ constructed in that proof acts as the smooth test function touching $u$ from above.

\subsection{The vanishing viscosity method}
The results of this paper can be used to study the regularity of solutions to nonlinear equations. In \cite{silvestre2009differentiability}, the H\"older estimates were applied to the derivative of the solution to a Hamilton-Jacobi equation with fractional diffusion. The author is also planning to apply the main theorem in this paper to study the regularity of conservation laws with fractional diffusion. For these nonlinear problems, one can use the vanishing viscosity approximation which makes the solution smooth (depending on the artificial viscosity) and then obtain some estimates that are independent of the viscosity coefficient. In this subsection, we show that our results apply in the vanishing viscosity setting.

For $\ep>0$, we can approximate the solution to \eqref{e:main} with the solution to
\begin{equation} \label{e:vanishing-viscosity}
u^\ep_t + b \cdot \grad u^\ep + (-\lap)^s u^\ep - \ep \lap u^\ep = f.
\end{equation}
Following the arguments in this article, one can prove that the solution $u^\ep$ is H\"older continuous with an a priori estimate independent of $\ep$.

The proof of Lemma \ref{l:point-estimate} applies to \eqref{e:vanishing-viscosity} as long as $\ep$ is bounded above by some constant $C$ (for example by one). Let us restate the lemma here so that it is suitable for the vanishing viscosity method.
\begin{lemma}[point estimate] \label{l:point-estimate-vanishing-viscosity}
Let $u$ be a $C^2$ function, $u \leq 1$ in $\R^n \times [-2,0]$ which satisfies the following inequality in $B_R \times [-2,0]$.
\begin{equation}
u_t - A |\grad u| + (-\lap)^s u - \ep \lap u \leq \ep_0 
\end{equation}
Where $R=2+A$ and $\ep<1$. Assume also that 
\[ \abs{\{u \leq 0\} \cap ( B_1 \times [-2 , -1]) } \geq \mu. \]
Then, if $\ep_0$ is small enough there is a $\theta>0$ such that $u \leq 1-\theta$ in $B_1 \times [-1 , 0]$.

\noindent (the maximal value of $\ep_0$ as well as the value of $\theta$ depend only on $s$ and $n$)
\end{lemma}

The proof is the same as the proof of Lemma \ref{l:point-estimate} except that at the moment when we find a contradiction for the equation at the maximum point of $w$ we also need to estimate $\ep \lap u$. We would have $- \ep \lap u(x_0,t_0) \geq \ep m(t) \lap_x \eta(x_0,t_0)$, and then we would argue that if $\eta$ is small enough $\lap \eta \geq 0$ or if $\eta$ is large, the term $C m(t) \eta$ controls $\ep m(t) \lap_x \eta(x_0,t_0)$ because $\ep < 1$.

Lemma \ref{l:point-estimate} is used at every scale in the iterative improvement of oscillation to prove Theorem \ref{t:main2}. When scaling $u$ to large values of $k$ in that proof the coefficient in front of the Laplacian may grow to values larger than one. The result of Theorem \ref{t:main2} is still valid for values of $\ep$ larger than zero. The proof would follow the following lines. First we do the same iteration to improve the oscillation in cylinders using Lemma \ref{l:point-estimate-vanishing-viscosity}. At some scale $k$, the coefficient $\ep$ in front of the Laplacian becomes larger than one for the first time. At that point we cannot apply Lemma \ref{l:point-estimate-vanishing-viscosity} anymore, but the equation is just a lower order perturbation of the heat equation, so it has H\"older estimates and the decay of the oscillation in the following cylinders follows because of that.

Alternatively, we can apply Lemma \ref{l:point-estimate-vanishing-viscosity} iteratively until we reach some scale $k$ that will be finer the smaller $\ep$ is. So, for every $\ep>0$, we can prove that $|u(x)-u(y)| \leq C|x-y|^\alpha$ for $|x-y|>r^{-k_0}$ just by iterating Lemma \ref{l:point-estimate-vanishing-viscosity} as in the proof of Theorem \ref{t:main2}. When we take $\ep \to 0$, this estimate will hold for all $x$ and $y$ since $k_0 \to \infty$ as $\ep \to 0$. This is the approach that was used in \cite{MR2600693}.

\subsection{Integro-differential diffusions}
The fractional Laplacian $(-\lap)^s$ in \eqref{e:main} could be replaced by general integro-differential operators of order $2s$. These operators have the form
\[ Iu(x) = \int_{\R^n} (u(x) - u(x+y)) \frac{k(x,y)}{|y|^{n+2s}} \dd y \]
The kernel $k(x,y)$ is assumed to be symmetric: $k(x,y) = k(x,-y)$ and bounded below and above: $\lambda \leq k(x,y) \leq \Lambda$. No modulus of continuity with respect to $x$ is necessary.

We now explain how to understand weak solutions of an equation of the form
\begin{equation} \label{e:equation-measurablecoefficients}
 u_t + b \cdot \grad u + \int_{\R^n} (u(x)-u(x+y)) \frac{k(x,y,t)}{|y|^{n+2s}} \dd y = f
\end{equation}
for some kernel $k$, some bounded right hand side $f$ and some vector field $b$ (either bounded or $C^{1-2s}$ depending on $s$). If $s \in (0,1/2)$, the vector field $b$ is continuous and the equation \eqref{e:equation-measurablecoefficients} implies that the following two inequalities hold:
\begin{align*}
 u_t + b \cdot \grad u - M^- u(x) &\geq -B \\
 u_t + b \cdot \grad u - M^+ u(x) &\leq B
\end{align*}
where $B = \norm{f}_{L^\infty}$, and $M^\pm$ stands for the maximal and minimal operators of order $2s$:
\[ M^\pm u(x) = \int_{\R^n} \frac{ \Lambda (u(x+y)+u(x-y)-2u(x))^{\pm} - \lambda (u(x+y)+u(x-y)-2u(x))^{\pm} }{|y|^{n+2s}} \dd y. \]
These operators were first defined in \cite{silvestre2006holder}. Even though it seems hard to make sense of \eqref{e:equation-measurablecoefficients} in any weak sense, the two inequalities above make perfect sense in the viscosity sense, and our H\"older estimates depend on those two inequalities only.

In the case $s \in [1/2,1)$ the vector field $b$ is discontinuous, so the inequalities above do not make sense in the viscosity sense. However, we can replace the term $b \cdot \grad u$ by $\pm A |\grad u|$ respectively, where $A = \norm{b}_{L^\infty}$, in which case the viscosity solution definition can be applied.

In \cite{silvestre2009differentiability}, the H\"older estimates were proved for a function $u$ satisfying both inequalities above in the viscosity sense in the case $s=1/2$.

\section*{Acknowledgment}

Luis Silvestre is partially supported by NSF grant DMS-1001629 and the Sloan fellowship.

\bibliographystyle{plain}   
\bibliography{hapd}
\index{Bibliography@\emph{Bibliography}}%

\end{document}